\documentclass{amsart}

\newtheorem{theorem}{Theorem}
\newtheorem{corollary}{Corollary}

\begin{document}

\title{A note on harmonic functions on surfaces}
\markright{Harmonic functions on surfaces}
\author{Jean C. Cortissoz}

\maketitle

\begin{abstract}
We review and give elementary
proofs of Liouville type properties of harmonic and subharmonic
functions in the plane endowed with a complete Riemannian metric, and
 prove a gap theorem for the possible growth of harmonic functions 
when this metric has nonnegative Gaussian curvature.
\end{abstract}

\section{Introduction.}

Given a function $u:\Omega\longrightarrow \mathbb{R}$, where $\Omega$ is an open subset of  $\mathbb{R}^2$,
the Laplacian is defined, in rectangular coordinates, as 
\[
\Delta u = \frac{\partial^2 u}{\partial x^2}+\frac{\partial^2 u}{\partial y^2},
\] 
and we say that $u$ is harmonic (resp. subharmonic) if $\Delta u =0$ ($\Delta u\geq 0$).
The classical Liouville's Theorem in $\mathbb{R}^2$ states that a bounded harmonic function is constant
(For a beautiful proof of this fact we recommend \cite{Nelson}). A stronger
version says that if $u$ is a subharmonic function bounded above then it must be constant; we refer
to this property of the plane as parabolicity, and, no doubt, it is an amazing fact that  
being a solution to a partial differential identity or a partial differential inequality may determine the growth properties of a function. 

Over the years, the analysis studied on $\mathbb{R}^n$ has been carried over to Riemannian manifolds,
a realm where a differential and an inner product structure coexist.
In particular,
a Laplacian operator acting over functions can be defined, and hence it is a framework 
in which the concept of
harmonic, subharmonic and superharmonic function have a natural extension. 
So it is also natural to ask which properties of harmonic functions,
such as Liouville's Theorem, or parabolicity 
are preserved in a Riemannian manifold. 

In this paper we are interested in Liouville type theorems and gap theorems
on surfaces with a pole (i.e., surfaces where polar coordinates can be defined). By a Liouville type property 
we mean a theorem that states that if a harmonic function is conveniently bounded then it must be constant,
and by a gap theorem we mean a theorem that imposes restriction on how fast a harmonic function must
grow so that it does not belong to a class of strictly lower growth. 

Before starting to throw definitions formulas
and theorems at the reader, let us mention some interesting results related to the work we will present in this paper.
Regarding Liouville type theorems, of 
great importance are the results of
Ahlfors and Milnor, which in the case of surfaces endowed with a rotationally symmetric
metric relates
an intrinsic quantitity, the curvature,
to the behavior of subharmonic functions, and to be more precise
to the parabolicity of the surface. Green and Wu (\cite{GreeneWu}) extended 
the Ahlfors-Milnor theorem to the case of surfaces with a pole.
We will give a relatively simple proof of part of the Ahlfors-Milnor-Greene-Wu parabolicity criterion, 
in which our main tool will be the Strong Maximum Principle.

Regarding gap properties for harmonic functions, on the classical side, that is,
in the complex plane, it is well known that if the rate of growth of a harmonic function is bounded by
a power of the distance to a fixed point, then it must be a polynomial. This is a consequence
of the analiticity of harmonic functions in $\mathbb{R}^2$ and of the Cauchy estimates. 
A most recent result has been proved by Ni and Tam in \cite{NiTam}: Here
the authors show how fast, in a K\"ahler manifold
of positive bisectional curvature, a superlinear harmonic function must grow. As a treat 
for the reader, we improve upon Ni and Tam's result 
in the case of a surface with a pole: This is the only new result in this paper 
(at least to the best of our knowledge). 

The main ideas of our proofs are contained in, and someone could
even say they are transplanted, via a clasical comparison
theorem due to Sturm, from the beautiful book ``Maximum Principles in Differential Equations'' 
\cite{ProtterWeinberger}, which we highly recommend.
 We also hope that this note serves as
an introduction, assuming as little as possible prerequisites from the reader, to the study of harmonic
functions in Riemannian geometry.

\section{preliminaries.}

For the convenience of the reader, let us give a quick review of a few concepts in Riemannian geometry
that we shall be using in what follows. 
We will consider $\mathbb{R}^2$ and fix polar coordinates
$\left(r,\theta\right)$ with respect to the origin, and we will endow it with a family
 of inner products of the form
\[
g=dr^2+\left(f\left(r,\theta\right)\right)^2d\theta^2,
\]
where $f:\left(0,\infty\right)\times \left[0,2\pi\right]\longrightarrow \left(0,\infty\right)$ 
is a smooth function such that 
\[
f\left(r,0\right)=f\left(r,2\pi\right),\quad
\lim_{r\rightarrow 0^+} f\left(r,\theta\right)=0\quad \mbox{and} \lim_{r\rightarrow 0^+}f'\left(r,\theta\right)=1,
\]
where we have used (and will use in what follows) $f'$ to denote differentiation with respect to $r$.

For those not familiar with Riemannian manifolds, $g$ represents a way of measuring vectors, and 
it is called a Riemannian metric. It defines
an inner product for vectors 
based at the point $\left(r,\theta\right)$, and represented in the
basis
\[
\frac{\partial}{\partial r}=\frac{x}{\sqrt{x^2+y^2}}\mathbf{i}+\frac{y}{\sqrt{x^2+y^2}}\mathbf{j},
\quad 
\frac{\partial}{\partial \theta} = y\mathbf{i}-x\mathbf{j}
\] 
in the following way. If we have 
\[
\mathbf{v_j}=a_j\frac{\partial}{\partial r}+b_j\frac{\partial}{\partial \theta},
\]
based at the point $\left(r,\theta\right)$ (the point $\left(r\cos \theta,r\sin\theta\right)$
in rectangular coordinates),
then
\[
g\left(\mathbf{v_1},\mathbf{v_2}\right)=a_1a_2+b_1b_2\left(f\left(r,\theta\right)\right)^2.
\]
Notice that with the choice $f\left(r,\theta\right)=r$ we obtain the usual inner product of vectors in the plane.
The pair $\left(M,g\right)$ is called a Riemannian surface (as opposed to a Riemann surface), and,
as discovered by Gauss, Riemannian surfaces have an important
intrinsic estimate: the curvature. 
From the expression for a Riemannian metric given above, the curvature can be computed as
\[
K_{g}\left(r,\theta\right)=-\frac{f''\left(r,\theta\right)}{f\left(r,\theta\right)}.
\]
In a Riemannian surface the gradient of a function $u:M\longrightarrow \mathbb{R}$ can be defined, 
due to the fact that given a nondegenerate scalar product a metric
dual of the derivative of a function can be defined. In our case then, the 
gradient of $u$ can be computed
as
\[
\nabla_g u = \frac{\partial u}{\partial r}\frac{\partial}{\partial r} + \frac{1}{f^2}\frac{\partial u}{\partial \theta}\frac{\partial}{\partial \theta}.
\]
We can also define a Laplacian, which is the operator of our utmost interest:
\[
\Delta_g = \frac{\partial^2}{\partial r^2}+\frac{f'}{f}\frac{\partial}{\partial r}+\frac{1}{f}\frac{\partial^2}{\partial \theta^2}
-\frac{f_{\theta}}{f^3}\frac{\partial}{\partial \theta},
\]
and here $f_{\theta}$ denotes $\displaystyle \frac{\partial f}{\partial \theta}$.

Given a Laplacian, we can define
a $C^2$ function $u$ as harmonic if $\Delta_g u=0$, subharmonic if $\Delta_g u\geq 0$ and superharmonic if
$\Delta_g u\leq 0$.

The term $h_g:=\frac{f'}{f}$ in the expression for the 
Laplacian is rather important for the following discussion. It gives the mean curvature of the circle of
radius $r$ with respect to the metric $g_0$. We will need to estimate this term, and the tool we will employ
is the following  comparison, due to Sturm, and which is known as the Laplacian Comparison Theorem
among geometers.
\begin{theorem}
Let $h, f: \left[a,\infty\right)\longrightarrow \left(0,\infty\right)$ be $C^2$ functions. If
$\displaystyle \frac{f''}{f}\left(r\right)\leq \frac{h''}{h}\left(r\right)$, $r>a$, and
$\displaystyle\frac{f'}{f}\left(a\right)\leq\frac{h'}{h}\left(a\right)$ then 
$\displaystyle\frac{f'}{f}\left(r\right)\leq\frac{h'}{h}\left(r\right)$ for $r>a$.
\end{theorem}  

The proof of this theorem is based upon the following observation:
\[
\left(\frac{w'}{w}\right)'=-\left(\frac{w'}{w}\right)^2+\frac{w''}{w}.
\]
Notice then that the hypothesis imply that 
$\displaystyle\left(\frac{f'}{f}\left(r\right)\right)'\leq\left(\frac{h'}{h}\left(r\right)\right)'$
whenever $\displaystyle \frac{f'}{f}\left(r\right)=\frac{h'}{h}\left(r\right)$.

Another important tool in the arguments that follow is the Maximum Principle.

\begin{theorem}
Let $\Omega$ be a bounded open subset of 
$\mathbb{R}^2$,
$\mathbb{R}^2$ endowed with
a Riemannian metric $g$, and 
let $u:\overline{\Omega}\longrightarrow \mathbb{R}$ be a 
$C^2\left(\Omega\right)\cap C\left(\overline{\Omega}\right)$ function.
Then we have that:

\noindent
(i) if $\Delta_g u\leq 0$ and $u$ has a minimum in the interior of $\Omega$ then $u$ 
is constant.

\noindent
(ii) if $\Delta u_g \geq 0$ and $u$ has a maximum in the interior of $\Omega$ then
$u$ is constant.
\end{theorem}

This theorem follows from the Strong Maximum Principle for elliptic operators 
(See Section 3.2 in \cite{GilbargTrudinger}, in particular Theorem 3.5).

\section{Bounded harmonic functions.}

Before we state the main result of
this section, we shall make an observation. Let $z:\left[a,\infty\right)\longrightarrow \left(0,\infty\right)$
be a $C^2$ function.
Direct calculation shows that the function
\[
h\left(r\right)=\int_a^r\frac{1}{z\left(\sigma\right)}\,d\sigma
\]
satisfies the identity
\[
\frac{d^2 h}{dr^2}+\frac{z'\left(r\right)}{z\left(r\right)}\frac{d h}{dr}=0,
\]
i.e, it is harmonic with respect to the metric
\[
g_0=dr^2+z\left(r\right)^2 d\theta^2,
\]
defined on the complement of the open ball of radius $a>0$. 

Hence if $g=dr^2+\left(f\left(r,\theta\right)\right)^2d\theta^2$ is a metric on $\mathbb{R}^2$ such that
$\displaystyle K_{g}\left(r,\theta\right)\geq -\frac{z''}{z}\left(r\right)$, then $\Delta_g h\geq 0$.
Indeed, notice that $\displaystyle \lim_{r\rightarrow a^+}\frac{z'}{z}\left(r\right)=\infty$, and hence
for $a_0$ close to $a$,
$\displaystyle \frac{f'}{f}\left(a\right)\leq \frac{z'}{z}\left(a\right)$,
so 
the hypothesis of Sturm's comparison theorem hold, and
 we have that $\displaystyle \frac{f'}{f}\left(r\right)\leq \frac{z'}{z}\left(r\right)$, for $r\geq a$. Therefore,
\[
\Delta_g h = \frac{d^2 h}{dr^2}+\frac{f'}{f}\frac{d h}{dr}
=\frac{d^2 h}{dr^2}+\frac{f'}{f}\frac{1}{z}\leq 
\frac{d^2 h}{dr^2}+\frac{z'}{z}\frac{1}{z}=0.
\]
Given a continuous function $v$ define
\[
M\left(v;r\right)=\sup_{\theta\in\left[0,2\pi\right)} \left|v\left(r,\theta\right)\right|.
\]
We are ready to show the following result.  
\begin{theorem}
\label{maintheorem}
Let $g=dr^2+f\left(r,\theta\right)^2d\theta^2$ be a metric on $\mathbb{R}^2$. Let $z$ be as above,
and assume  that $K_g\left(r,\theta\right)\geq -\frac{z''}{z}\left(r\right)$ for $r$ large enough. Then,
any subharmonic function $u$ which satisfies
\[
\liminf_{r\rightarrow\infty} \frac{M\left(u;r\right)}{h\left(r\right)}=0
\]
must be constant.
\end{theorem}

\begin{proof}
 In 
what follows we will denote by $B_R$ the ball of radius $R$ centered
at $\left(0,0\right)$. Fix $R_1>0$, fix  $\delta>0$ also small, and define the function
\[
w_{\delta, \eta}=u-\delta h\left(r\right)-M\left(u;R_1\right).
\]
By hypothesis, we can take $R_2$ larger than $R_1$ and such that
$w_{\delta,\eta}\leq 0$ on both $\partial B_{R_1}$ and $\partial B_{R_2}$. It is also clear 
from its definition that
$\Delta_g w_{\delta,\eta}\geq 0$.
From the Maximum Principle 
it follows that
\[
w_{\delta,\eta}\leq 0
\]
on the annulus of inner radius $R_1$ and outer radius $R_2$. Since $\delta>0$ can be made arbitraryly small, we conclude that for all
$P$ in the annulus the estimate 
\[
\left|u\left(P\right)\right|\leq M\left(u;R_1\right)
\]
holds, and hence that $\left|u\right|\leq M_{R_1}$  on the ball of  radius $R_2$, and from this via the Maximum Principle 
(as $u$ attains its maximum at an interior point of $B_{R_2}$) 
we can conclude that $u$ is constant in the ball of radius $R_2$. Since $R_2$ can be taken arbitrarily large, the
theorem follows.

\end{proof}

The proof given above is presented in \cite{ProtterWeinberger} in the case of a flat metric 
(See Theorem 19 in Section 2 of \cite{ProtterWeinberger} and the example thereafter), 
but its generalization
to other metrics is straightforward. Besides,
the previous theorem has the following interesting consequence:
\begin{corollary}
Let $u$ be a harmonic function on a surface such that for $r\geq r_0\geq 1$ its
curvature function satisfies $K_g\geq -\frac{1}{r^2\log r}$. If
\[
\liminf_{r\rightarrow \infty}\frac{M\left(u;r\right)}{\log\log r}=0
\]
then $u$ is constant. In particular, if $u$ is bounded then it is constant.
\end{corollary}

This corollary is due to Greene and Wu (Theorem D in \cite{GreeneWu}),
and its proof is quite simple: take $\displaystyle z\left(r\right)= r\log \left(\frac{r}{r_0}\right)$ for $r\geq r_0$.
Notice that, since $r_0\geq 1$, then 
\[
-\frac{z''}{z}=-\frac{1}{r^2\log\left(\frac{r}{r_0}\right)}\leq -\frac{1}{r^2\log r}\leq K_g.
\]
On the other hand, $\lim_{r\rightarrow r_0^+} \frac{z'}{z}\left(r\right)=+\infty$, and hence by Sturm's comparison theorem,
$\Delta_g h\geq 0$, with $h$ as defined above; by Theorem \ref{maintheorem} the result follows.

We want to point out that in higher dimensions, Liouville's Theorem has been extended by Yau, via his gradient estimate (see below), to manifolds
of nonnegative Ricci curvature. 

\subsection{On the Ahlfors-Milnor-Greene-Wu parabolicity criterion: an application of  Hadamard's three circles theorem.}

A surface $M$ is called parabolic if any subharmonic function (i.e. $\Delta u\geq 0$) bounded above is constant.
Milnor in \cite{Milnor} showed that given a rotationally symmetric metric on $\mathbb{R}^2$, it is
parabolic if for large enough $r$ the curvature of the metric is larger than or equal to
$\displaystyle -\frac{1}{r^2\log r}$.

We use the Maximum Principle to give a proof of a generalization
of Milnor's criterion (which is a theorem due to Greene and Wu).
We assume that we have endowed $\mathbb{R}^2$ with a metric of the
form $\displaystyle dr^2+\left(f\left(r,\theta\right)\right)^2d\theta^2$ whose
curvature is larger than or equal to $\displaystyle -\frac{1}{r^2\log r}$.

First, we have the following version of Hadamard's three circles theorem. Let $u$ be a subharmonic function, and let 
\[
M_r=\max_{\theta\in \left[0,2\pi\right)}u\left(r,\theta\right).
\]
Notice that if $r_2>r_1$ then $M_{r_2}\geq M_{r_1}$, by the Maximum Principle. So let $r_1<r<r_2$, and define
\[
\varphi\left(r\right)=\frac{1}{\log\left(\frac{\log r_2}{\log r_1}\right)}
\left(M_{r_1}\log\left(\frac{\log r_2}{\log r}\right)+M_{r_2}\log\left(\frac{\log r}{\log r_1}\right)\right).
\]
It is easy to check that $\Delta \varphi \leq 0$, and hence $u-\varphi$ is subharmonic, and also
$u-\varphi\leq 0$ on both $\partial B_1$ and $\partial B_2$. We can conclude via the Maximum Principle
that
\[
M_r\leq \varphi\left(r\right).
\]
This last inequality is our version of Hadamard's three circles theorem. Now assume that $u$ is bounded above.
By taking  
$r_2\rightarrow \infty$, we obtain the estimate
\[
M_r\leq M_{r_1}.
\]
But then, since $r>r_1$, and $r$ is arbitrary by the Maximum Principle $u$ must be constant. We can conclude that
the surface is parabolic. Again, we must point out that our proof follows closely
the arguments given in Chapter 2, Section 12 in \cite{ProtterWeinberger}.

\section{A gap theorem for surfaces with nonnegative Gaussian curvature.}

It is an exercise in complex analysis to prove the following result. Given a holomorphic function $f$
such that
$\left|f\left(z\right)\right|\leq C\left|z\right|^k+B$ then $f$ is a polynomial of degree at most $\lfloor k\rfloor$. 
From this we can conclude that 
there are gaps between the possible growth that a holomorphic, and in consequence, harmonic functions
can have. For instance, a harmonic function of subquadratic growth (i.e., $k<2$) must be of at most linear growth
(i.e., $k$ must be less or equal to $1$). 

In this section we prove a gap theorem on the possible growth of harmonic 
functions on a complete noncompact surface $M$ of positive Gaussian curvature. To be more precise 
we will show the following gap theorem.
\begin{theorem}
\label{mainresult}
Let $g=dr^2+f\left(r,\theta\right)^2d\theta^2$ be a metric on $\mathbb{R}^2$ with nonnegative Gaussian
curvature. Let $u:M\longrightarrow \mathbb{R}$ be a harmonic function. 
If for all $\delta>0$
\[
\liminf_{r\rightarrow\infty} \frac{M\left(u;r\right)}{r^{1+\delta}}=0,
\]
then there is a constant $C>0$ such that $\left|u\left(r,\theta\right)\right|\leq Cr$ (i.e., it must be of
linear growth).
\end{theorem}

Let us comment on the significance of Theorem \ref{mainresult} in view of what is known. Ni and Tam in \cite{NiTam} showed
that on a K{\"a}hler manifold if a harmonic function satisfies that
\[
\limsup_{r\rightarrow \infty}\frac{u\left(r,\theta\right)}{r^{1+\delta}}=0
\]
for all $\delta>0$ then it must be of linear growth. This estimate is sharp in the following sense:
It is known that there are complete noncompact surfaces of nonnegative Gaussian curvature that support
harmonic functions which grow like $r^{1+\delta}$ for $0<\delta<1$. Theorem \ref{mainresult} 
can be compared to Ni and Tam's theorem since in dimension 2 every orientable Riemannian manifold is K\"ahler.

\subsection{Proof of Theorem \ref{mainresult}.}

First we must recall, without a proof (one of which is via the Maximum Principle), a classical estimate due 
to Yau (see \cite{STYau}) for the gradient of a harmonic function. To simplify the notation, in what follows, we will suppress the
subindex $g$ indicating the dependence on the metric of the gradient and Laplace operator.

\begin{theorem}
Let $u:\Omega \subset \mathbb{R}^2\longrightarrow \mathbb{R}$, be a positive harmonic function, $\mathbb{R}^2$
being endowed with a Riemannian metric of positive curvature. Then the estimate
\[
\left|\nabla \log u\right|\leq \frac{C}{r}
\]
holds on any ball of radius $r$ contained in $\Omega$.
\end{theorem} 

Before engaging in the proof of Theorem \ref{mainresult}, we shall show a weaker result. 
To this end, let us introduce
Bochner's identity
\[
\frac{1}{2}\Delta \left|\nabla u\right|^2=\left|Hess\, u\right|^2+g\left(\nabla \Delta u, \nabla u\right)+K_gg\left(\nabla u,\nabla u\right).
\]
We recommend the reader non familiar with this formula to prove it in the case of $\mathbb{R}^n$, where
the term involving the curvature does not appear, and $g$ is the usual inner product. Recall that
in this case $Hess \, u$, the Hessian of $u$, is the matrix of second derivatives of $u$. The general case in a Riemannian manifold
follows from the nonconmutativity of the covariant derivatives, which is measured by the curvature
(for 
a proof of this formula see Lemma 1.36 and Exercise 1.37 in \cite{ChowLuNi}, and beware that
$Hess\, u= \nabla\nabla u$).

Now observe that if $u$ is harmonic and 
$K_g\geq 0$, then 
from Bochner's identity follows that 
$\left|\nabla u\right|^2$ is subharmonic. Let us assume that 
\[
\liminf_{r\rightarrow\infty} \frac{M\left(u;r\right)}{r\sqrt{\log r}}=0.
\]
Then, by Yau's estimate, 
\[
\liminf_{r\rightarrow\infty} \frac{M\left(\left|\nabla u\right|^2;r\right)}{\log r}=0
\]
so $\left|\nabla u\right|^2$ is constant
by Theorem \ref{maintheorem} (take $z\left(r\right)=r$, and the hypothesis hold since $K_g\geq 0$), and hence
$u$ must be of linear growth. 

However we can do much better: It
is time to give a proof of Theorem \ref{mainresult}. Following the work of Ni and Tam, instead
of using Bochner's identity,
 we will make use of the following 
identity, which is valid for any harmonic function $u$ defined on a surface:
\begin{equation}
\label{Bochner2}
\Delta \log \left(1+\left|\nabla u\right|^2\right)=
\frac{2\left|Hess \, u\right|^2+2K_g\left|\nabla u\right|^2\left(1+\left|\nabla u\right|^2\right)}{\left(1+\left|\nabla u\right|^2\right)^2}.
\end{equation}
Not being as well known as 
Bochner's, we shall give a proof of this identity in the last paragraphs of this paper, so let us then continue
with the proof of Theorem \ref{mainresult}.

Again, identity (\ref{Bochner2}) implies in the case of nonnegative curvature that $\log \left(1+\left|\nabla u\right|^2\right)$ is
subharmonic. Now use the hypothesis: If for $\delta>0$ 
\[
\liminf_{r\rightarrow \infty}\frac{M\left(u;r\right)}{r^{1+\delta}}=0,
\]
then by Yau's gradient estimate there is a sequence $r_{\delta,k}\rightarrow \infty$ such that
\[
\left|\nabla u\right| \leq Cr_{\delta,k}^{\delta}
\]
and hence,
\[
\log \left(1+\left|\nabla u\right|^2\right)\leq C\log{r_{\delta,k}^{\delta}}.
\]
Since $\delta>0$ can be as small as we wish, as a consequence we can deduce that
\[
\liminf_{r\rightarrow \infty} \frac{M\left(\log \left(1+\left|\nabla u\right|^2\right),r\right)}{\log r}=0
\]
and from Theorem \ref{maintheorem} we can conclude that $\log \left(1+\left|\nabla u\right|^2\right)$ is constant,
and hence that $\left|\nabla u\right|$ is constant, i.e., $u$ is of linear growth.

\subsection{Proof of formula (\ref{Bochner2}).}
As is customary,
we pick a local orthonormal frame $e_1,e_2$ around the point
where we will be performing our computations. A subindex $i$ will denote covariant
differentiation with respect to (or in the direction of, as you prefer) $e_i$. 
We shall enforce Einstein's summation condition, i.e., we add over repeated subindices.
For instance, we have:
\[
\Delta u=u_{jj}=u_{11}+u_{22}, \quad \mbox{and} \quad
\left|Hess\,u\right|^2=u_{ij}u_{ij}=u_{11}^2+u_{12}^2+u_{21}^2+u_{22}^2.
\]

We are ready to begin:
\begin{eqnarray*}
\left(\log\left(1+\left|\nabla u\right|^2\right)\right)_{jj}
&=&
\left(\frac{2}{1+\left|\nabla u\right|^2}u_{ij}u_i\right)_j\\
&=& -\frac{4}{\left(1+\left|\nabla u\right|^2\right)^2}u_{ij}u_{kj}u_iu_k\\
&&+\frac{2}{1+\left|\nabla u\right|^2}u_{ij}u_{ij}
 +\frac{2}{1+\left|\nabla u\right|^2}u_{ijj}u_i.
\end{eqnarray*}
The term $\displaystyle u_{ij}u_{kj}u_{i}u_{k}$ when written in expanded form is
\begin{equation*}
u_{11}^2u_1^2+2u_{12}u_{11}u_1u_2+u_{12}^2u_2^2+u_{22}^2u_2^2+2u_{21}u_{22}u_1u_2+u_{21}^2u_1^2,
\end{equation*}
which by using that $u_{11}=-u_{22}$ and $u_{21}=u_{12}$, reduces to
\[
u_{11}^2u_1^2+u_{12}^2u_2^2+u_{22}^2u_2^2+u_{21}^2u_1^2.
\]
Now we need to be a little bit careful. Notice that 
\begin{equation*}
u_{11}u_i= -u_{22}u_i,
\end{equation*}
so we get that
\begin{equation*}
u_{ii}^2\left|\nabla u\right|^2=u_{11}^2u_1^2+u_{22}^2u_2^2,
\end{equation*}
and hence
\[
\left(u_{11}^2+u_{22}^2\right)\left|\nabla u\right|^2=2\left(u_{11}^2u_1^2+u_{22}^2u_2^2\right),
\]
which leads to
\[
\left(1+\left|\nabla u\right|^2\right)u_{ij}u_{ij}=2\left(u_{11}^2u_1^2+u_{22}^2u_2^2\right)
+2u_{12}^2u_1^2+2u_{12}^2u_2^2 +u_{ij}u_{ij},
\]
from which we obtain
\[
-\frac{4}{\left(1+\left|\nabla u\right|^2\right)^2}u_{ij}u_{kj}u_iu_k
+\frac{2}{1+\left|\nabla u\right|^2}u_{ij}u_{ij}
=
\frac{2u_{ij}u_{ij}}{\left(1+\left|\nabla u\right|^2\right)^2}.
\]

It is from $\displaystyle u_{ijj}u_{i}$ that we obtain the curvature term. 
Indeed, the Ricci identity (Lemma 1.36 in \cite{ChowLuNi}) tells us that
\[
\Delta u_j= \left(\Delta u\right)_j+R_{ij}u_i,
\]
where $R_{ij}$ is the Ricci tensor of the metric. In the case of a surface, $R_{ij}=K_g g_{ij}$, 
$K_g$ being the Gaussian curvature,
and hence we have
that
\[
\frac{2}{1+\left|\nabla u\right|^2}u_{ijj}u_i=\frac{2}{1+\left|\nabla u\right|^2}R_{ki}u_ku_i=
\frac{2K_g\left|\nabla u\right|^2}{1+\left|\nabla u\right|^2},
\]
and formula (\ref{Bochner2}) follows.




\end{document}